\documentclass{amsproc}

\usepackage{amsmath}
\usepackage{amsthm}
\usepackage{amssymb}
\usepackage{amscd}
\usepackage{graphicx}
\usepackage{epsfig}
\usepackage{mathtools}
\usepackage{multirow}
\usepackage[all, knot]{xy}

\newcommand{\Z}{\mathbb{Z}}

\newtheorem{theorem}{Theorem}[section]
\newtheorem{corollary}[theorem]{Corollary}
\newtheorem{lemma}[theorem]{Lemma}
\newtheorem{proposition}[theorem]{Proposition}

\newtheorem{definition}[theorem]{Definition}

\newtheorem{example}[theorem]{Example}

\newtheorem{remark}[theorem]{Remark}

\input{epsf.sty}

\newcommand{\K}{\mathbb{K}}
\newcommand{\as}{\mathfrak{as}}

\begin{document}

\title[Hom-quasi-bialgebras]{Hom-quasi-bialgebras  }

\author[M. Elhamdadi]{Mohamed Elhamdadi}
\address{University of South Florida, USA}
\curraddr{}
\email{emohamed@math.usf.edu}
\thanks{}

\author[A. Makhlouf]{Abdenacer Makhlouf}
\address{Universit\'e de Haute Alsace, France}
\curraddr{}
\email{Abdenacer.Makhlouf@uha.fr}
\thanks{The second author would like to thank the Math department of University of South Florida for their hospitality during his visit.}


\date{}

\begin{abstract}
The purpose of this paper is to give a general  survey of Hom-bialgebras, which are  bialgebra-type structures where the identities are twisted by a morphism, and to extend the concept of quasi-bialgebra to Hom-setting. We provide some key constructions and generalize the concept of gauge transformation and Drinfeld's twists to  this type of generalized quasi-bialgebras. Moreover, we provide an example of twisted quantum double.  
\end{abstract}

\maketitle

\section{Introduction}
The first instances of twisted algebras by modifying identities appeared first in Mathematical physics were $q$-deformations of Witt and Virasoro algebras using $\sigma$-derivations lead to a twisted  Jacobi identity, see for example \cite{AizawaSaito,ChaiIsKuLuk,CurtrZachos1,Kassel1}. This structure, called  Hom-Lie algebras, were investigated by Hartwig, Larsson and Silvestrov  in   \cite{HLS,LS1}. The corresponding associative type objects, called Hom-associative algebras were introduced by Makhlouf and Silvestrov in \cite{MakSil-def}. The enveloping algebras of Hom-Lie algebras were studied by Yau in \cite{Yau:EnvLieAlg}. The dual notions, Hom-coalgebras, Hom-bialgebras, Hom-Hopf algebras and Hom-Lie coalgebras,  were developed first in \cite{HomHopf,HomAlgHomCoalg}. Then the  study was  enhanced  in \cite{Yau:YangBaxter,Yau:comodule,Yau:HQuantumGroup1} and done with a categorical point of view in \cite{Canepl2009}.  Further developments and results about Hom-algebras could be found in \cite{AEM,AmmarMakhloufJA2010,FregierGohrSilv,Gohr,Yau:homology,Yau:YangBaxter2}.
 The main feature of Hom-algebra and Hom-coalgebra structures is that the classical identities are twisted by a  linear self map.

Quasi-Hopf algebras were introduced by Drinfeld  \cite{Drinfeld,DRIN1,DRIN2}. They are generalizations of Hopf algebras in which the coassociativity is replaced by a weaker condition. The formulation of enveloping algebra equipped with quasi-Hopf structure obtained Knizhnik-Zamolodchikov equations arises naturally in conformal field theory and in the theory of Vassiliev invariants. The quantum double introduced by Drinfeld is one the most important of quantum group constructions.  Many of the ideas of constructions from the theory of Hopf algebras have analogues in the quasi-Hopf algebra setting. Quasi-triangular quasi-Hopf algebra quantum double was introduced by Majid using categorical Tannaka-Krein reconstruction methods \cite{ Majid,MajidPink}. He also introduced the concept of dual quasi-Hopf algebra structure. For further developments on quasi-Hopf algebras see \cite{BulacuPanaite,BulacuCaneTorr,EtingofGelaki,Kabook,Schauenb1,Schauenb2,Schauenb3}.

The aim of this paper is to give an overview of the theory of Hom-Hopf algebras  and to introduce and study the notion of Hom-quasi-bialgebra  (HQ-bialgebra). We develop, for this type of algebras, the concept of gauge transformation and Drinfeld's twist construction. This Drinfeld's twisting construction ensure that any Hom-bialgebra could be twisted into a HQ-bialgebra  and that the class of HQ-bialgebras is itself closed under the twisting operation. 

The paper is organized as follows.
Section~\ref{Review} reviews the basic definitions and properties of
Hom-bialgebras and  Hom-Hopf 
 algebras. In Section~\ref{Quasi}, we introduce the notion of Hom-quasi-bialgebra and give  a construction which deform any  given HQ-bialgebra with a HQ-bialgebra morphism to a new HQ-bialgebra. In particular, any quasi-bialgebra with a quasi-bialgebra morphism gives rise to a HQ-bialgebra. Moreover, we  define quasi-triangular HQ-bialgebra. Section~\ref{DrinfeldTwist} deals with gauge transformation. We extend the Drinfeld's twist construction to HQ-bialgebras. An example of twisted quantum doubles is provided in Section~\ref{ExampleDrinfeldTwist}.

\section{Overview of Hom-associative algebras, Hom-coalgebras and Hom-bialgebras}\label{Review}
In this section we summarize definitions and describe some of basic
properties of Hom-algebras, Hom-coalgebras, Hom-bialgebras and Hom-Hopf algebras
which generalize the classical algebra, coalgebra, bialgebra and Hopf algebra
structures.

Throughout  this paper $\K$ denotes a ground field  and $A$ is a $\K$-module. In the sequel we denote by  $\tau_{ij}$  linear maps $\tau :A^{\otimes n}\rightarrow A^{\otimes n}$ 
where $\tau_{i j}(x_1\otimes\cdots\otimes x_i\otimes \cdots \otimes x_j \otimes \cdots\otimes x_n) = (x_1\otimes\cdots\otimes x_j\otimes \cdots \otimes x_i \otimes \cdots\otimes x_n)$.\\

We mean by a Hom-algebra  a triple $(A, \mu, \alpha)$ in which   $\mu : A^{\otimes 2} \to A$ is a linear map and $\alpha : A \to A$ is a linear self map. The linear map  $\mu^{op} : A^{\otimes 2} \to A$ denotes the opposite map, i.e. $\mu^{op} = \mu\circ\tau_{12} $.  A  Hom-coalgebra is a triple $(A, \Delta, \alpha)$ in which   $\Delta :A  \to A^{\otimes 2}$ is a linear map and $\alpha : A \to A$ is a linear self map. The linear map  $\Delta^{op} : A  \to A^{\otimes 2}$ denotes the opposite map, i.e. $\Delta^{op} = \tau_{12} \circ \Delta$.
For a linear self-map $\alpha : A \to A$, we denote by $\alpha^n$ the $n$-fold composition of $n$ copies of $\alpha$, with $\alpha^0 \cong id
$. A Hom-algebra $(A, \mu, \alpha)$ (resp. a Hom-coalgebra $(A, \Delta, \alpha)$) is said to be \emph{multiplicative}  if $\alpha \circ \mu = \mu \circ \alpha^{\otimes 2}$ (resp. $\alpha ^{\otimes 2}\circ \Delta = \Delta \circ \alpha$). The Hom-algebra is called \emph{commutative} if $\mu = \mu^{op}$ and the Hom-coalgebra is called \emph{cocommutative} if $\Delta = \Delta^{op}$.
  Classical algebras or  coalgebras are also regarded as a Hom-algebras or Hom-coalgebras with identity twisting map. Given a Hom-algebra $(A, \mu, \alpha)$, we often use the abbreviation $\mu(x, y) = xy$ for $x, y \in A$. Likewise, for a Hom-coalgebra $(A,\Delta,\alpha)$, we will use Sweedler's notation $\Delta(x) = \sum_{(x)} x_1 \otimes x_2$ but often omit the symbol of summation. When the Hom-algebra (resp. Hom-coalgebra) is multiplicative, we also say that $\alpha$ is multiplicative for $\mu$ (resp.  $\Delta$). \\

\begin{definition}
A  \emph{Hom-associative algebra} is a Hom-algebra $(A, \mu, \alpha)$ satisfying 
\begin{equation}
 \mu \circ (\mu \otimes \alpha) =  \mu \circ (\alpha \otimes \mu).
\end{equation}

 A Hom--associative algebra $A$ is called  \emph{unital} if there exists a linear map $\eta: \K\to A$  such that 
\begin{equation}\label{unit}
\mu\circ\left(  id
_{A}\otimes\eta\right)  =\mu\circ\left(  \eta\otimes
id
_{A}\right)  =\alpha.
\end{equation}
The unit element is $1_A=\eta(1_\K)$. We refer to a unital Hom-associative algebra with a quadruple  $(A, \mu, \alpha,\eta)$.
 \end{definition}
We call Hom-associator the linear map $
\as_{A}$ defined by $\mu \circ (\mu \otimes \alpha - \alpha \otimes \mu)$. In term of elements, we have 
$\as_{A}(x, y, z) =\ (xy)\alpha(z) - \alpha(x)(yz),$ for $x, y, z \in A$.

We recover the usual definitions of  associator and associative algebra  when the twisting map $\alpha$ is the identity map .

Given a unital Hom-associative algebra. An element $x\in A$ is said to be \emph{invertible} and $x^{-1}$  is its inverse if $$x x^{-1}= x^{-1} x=1_A.$$
Notice that $x y=z$ implies $\alpha^2 (x)=z \alpha(y^{-1}).$
Indeed, if $x y=z$ then $(x y)\alpha(y^{-1})=z\alpha(y^{-1})$. By Hom-associativity, we have $\alpha (x)( yy^{-1})=z\alpha(y^{-1})$ which is equivalent to $\alpha (x)1_A=z\alpha(y^{-1})$ and then to $\alpha^2 (x)=z \alpha(y^{-1}).$

Let $\left( A,\mu,\alpha \right)$ and $ \left( A^{\prime},\mu^{\prime}, \alpha^{\prime}\right)$ be two Hom-algebras. A linear map $f\ : A \rightarrow A^{\prime}$ is a \emph{ Hom-algebras morphism } if
\[
\mu^{\prime} \circ (f \otimes f) = f \circ \mu \quad \text{ and } \qquad f \circ \alpha = \alpha^{\prime}\circ f.
\]
It is said to be a \emph{weak Hom-algebras morphism} if holds only the first condition. If furthermore, the algebras are unital then $f(1_A)=1_{A^{\prime}}.$

\begin{example}\label{example1ass}
Let $\{x_1,x_2,x_3\}$  be a basis of a $3$-dimensional vector space
$A$ over $\K$. The following multiplication $\mu$ and linear map
$\alpha$ on $A=\K^3$ define a Hom-associative algebra over $\K${\rm:}
\[
\begin{array}{ll}
  \begin{array}{lll}
  \mu ( x_1,x_1)&=& a x_1, \ \\
  \mu ( x_1,x_2)&=&\mu ( x_2,x_1)=a x_2,\\
  \mu ( x_1,x_3)&=&\mu ( x_3,x_1)=b x_3,\\
  \end{array}
 & \quad
  \begin{array}{lll}
  \mu ( x_2,x_2)&=& a x_2, \ \\
  \mu ( x_2, x_3)&=& b x_3, \ \\
  \mu ( x_3,x_2)&=& \mu ( x_3,x_3)=0,
  \end{array}
\end{array}
\]

\[
\alpha(x_1) = a x_1, \quad
\alpha(x_2) = a x_2, \quad
\alpha(x_3) = b x_3,
\]
where $a,b$ are parameters in $\K$. Over a field of characteristic $0$, this algebra is not associative when $a\neq b$ and $b\neq 0$, since
\[
\mu(\mu(x_1,x_1),x_3))- \mu(x_1,\mu(x_1,x_3))=(a-b)b x_3.
\]
\end{example}

The categories of Hom-associative algebras  is closed under twisting self-weak morphisms.
The following Theorem  gives an easy way to provide new Hom-associative algebras with a given Hom-associative algebra and an algebra morphism. Then  one may  twist classical  associative algebra structure to  Hom-associative algebra structure. This procedure is called a twisting principle.

\begin{theorem}[\cite{Yau:homology,Yau:Poisson}] \label{twist-ass_Lie}
 Let $A= (A, \mu, \alpha,\eta)$ be a unital  Hom-associative algebra  and $\beta$ be a weak morphism, then $A_\beta = (A, \mu_\beta = \beta \circ \mu, \beta\circ\alpha,\eta)$ is a unital  Hom-associative algebra.

 In particular, let $A = (A, \mu,\eta)$ be a unital associative algebra and $\alpha : A \to A$ be an algebra morphism, i.e. $\alpha \circ \mu = \mu \circ \alpha^{\otimes 2}$. Then $A_\alpha = (A, \mu_\alpha = \alpha \circ \mu, \alpha,\eta)$ is a unital Hom-associative algebra.
\end{theorem}

\begin{proof}
We  do the proof in the particular case. We have
\begin{equation*}
\begin{split}
\as_{A_\alpha} &= \mu_\alpha \circ (\mu_\alpha \otimes \alpha - \alpha \otimes \mu_\alpha) \\
&= \alpha \circ \mu \circ \alpha^{\otimes 2} \circ (\mu \otimes id
 - id
 \otimes \mu) \\
&= \alpha^2 \circ \as_{A} = 0.
\end{split}
\end{equation*}
Since $(A,\mu)$ is associative, so $A_\alpha$ is a Hom-associative algebra.

Moreover the axiom \eqref{unit} is satisfied, indeed
\begin{equation*}
\beta\circ \mu\circ\left(  id
_{A}\otimes\eta\right)  =\beta\circ\mu\circ\left(  \eta\otimes
id
_{A}\right)  =\beta\circ\alpha.
\end{equation*}
\end{proof}

In the following we give an overview of the theory of Hom-bialgebras which was introduced in \cite{HomHopf,HomAlgHomCoalg} and enhanced in \cite{BEM2012,Canepl2009,Yau:YangBaxter,Yau:comodule,Yau:HQuantumGroup1}.
\begin{definition}

A \emph{Hom-coassociative coalgebra} is a Hom-coalgebra $\left( A,\Delta ,\alpha \right) $  satisfying
\begin{equation}\label{C1}
\left(\alpha \otimes \Delta \right) \circ \Delta = \left( \Delta \otimes \alpha\right) \circ \Delta.
\end{equation}
A Hom-coassociative coalgebra is said to be \emph{counital} if there exists a map $\epsilon
 :A\rightarrow \K$ satisfying
\begin{equation}\label{C2}
\left( id \otimes \epsilon
 \right) \circ \Delta = \alpha \qquad \text{ and }\qquad
\left( \epsilon
 \otimes id\right) \circ \Delta = \alpha.
\end{equation}
We refer to a counital Hom-coassociative coalgebra with a quadruple $\left( A,\Delta ,\alpha,\epsilon \right) $
\end{definition}

Let $\left( A,\Delta ,\alpha \right) $
and $\left( A^{\prime },\Delta^{\prime
},\alpha^{\prime }\right)
$ be two Hom-coalgebras (resp.  Hom-coassociative coalgebras). A
linear map  $f\ :A\rightarrow A^{\prime }$ is a \emph{morphism of  Hom-coalgebras} (resp. \emph{Hom-coassociative coalgebras}) if
\[
(f \otimes f)\circ \Delta = \Delta^{\prime} \circ f
\quad \text{ and} \quad f \circ\alpha = \alpha^{\prime}\circ f.
\]
It is said to be a \emph{weak Hom-coalgebras morphism } if holds only the first condition. If furthermore the Hom-coassociative coalgebras admit counits $\epsilon
$ and $\epsilon
^{\prime}$, we have moreover $\epsilon
 = \epsilon
^{\prime}\circ f $.

The category of coassociative Hom-coalgebras is closed under weak Hom-coalgebra morphisms. 
\begin{theorem} \label{thmConstrHomCoalg}
Let $(A,\Delta,\alpha,\epsilon
)$ be a counital Hom-coassociative coalgebra  and $\beta : A\rightarrow A$ be a weak Hom-coalgebra morphism. Then $(A,\Delta_\beta= \Delta \circ \beta,\alpha \circ \beta,\epsilon
)$  is a counital Hom-coassociative coalgebra.

In particular, let $(A,\Delta,\epsilon
)$ be a coalgebra and $\beta : A \to A$ be a coalgebra morphism. Then $(A,\Delta_\beta,\beta,\epsilon
)$ is a counital Hom-coassociative coalgebra.
\end{theorem}

\begin{proof}
We show that $(A,\Delta_\beta,\alpha\circ\beta,\epsilon
)$ satisfies the axiom \eqref{C1}.

Indeed, using the fact that $(\beta \otimes \beta) \circ \Delta = \Delta \circ \beta$, we have
\begin{align*}
\left( \alpha\circ\beta\otimes \Delta_\beta \right)
\circ \Delta_\beta &= \left(\alpha\circ\beta \otimes \Delta\circ\beta\right) \circ \Delta\circ\beta \\
&=(\left( \alpha \otimes
\Delta\right) \circ\Delta)\circ \beta^2\\
&=(\left( \Delta\otimes\alpha
\right) \circ\Delta)\circ \beta^2\\
&=\left(\Delta_\beta \otimes\alpha\circ\beta
\right) \circ \Delta_\beta.
\end{align*}
Moreover, the axiom \eqref{C2} is also satisfied, since we have
\begin{equation*}
\left( id \otimes \epsilon
 \right) \circ \Delta_\beta = \left( id \otimes \epsilon
 \right) \circ \Delta \circ \beta = \alpha \circ \beta = \left( \epsilon
 \otimes id\right) \circ \Delta \circ \beta = \left( \epsilon
 \otimes id\right) \circ \Delta_\beta.
\end{equation*}
\end{proof}

\begin{example}
 Theorem \ref{thmConstrHomCoalg} leads to the following examples:

 Let $(A,\Delta,\alpha)$ be a multiplicative Hom-coassociative coalgebra. For any nonnegative integer $n$,  $(A,\Delta_{\alpha^n}, \alpha^{n+1})$ is a Hom-coassociative coalgebra.

\end{example}

We show that there is a duality between  Hom-associative and  Hom-coassociative structures.


\begin{theorem}
Let $(A,\Delta,\alpha)$ be a  Hom-coassociative coalgebra. Then its dual vector space is provided with a structure of Hom-associative  algebra $(A^*,\Delta^*,\alpha^*)$, where $\Delta^*,\alpha^*$ are the transpose map. Moreover, the Hom-associative  algebra is  unital  whenever  $A$ is counital.
\end{theorem}

\begin{proof}
The product $\mu = \Delta^*$ is defined from $A^* \otimes A^*$ to $A^*$ by
\begin{equation*}
(fg)(x) = \Delta^*(f,g)(x) = \langle \Delta(x),f \otimes g \rangle = (f \otimes g)(\Delta(x)) = \sum_{(x)} f(x_1)g(x_2), \quad \forall x \in A
\end{equation*}
where $\langle \cdot,\cdot \rangle$ is the natural pairing between the vector space $A \otimes A$ and its dual vector space. For $f,g,h \in A^*$ and $x \in A$, we have
\begin{gather*}
(fg) \alpha^*(h)(x) = \langle (\Delta \otimes \alpha) \circ \Delta(x),f \otimes g \otimes h \rangle
\shortintertext{and}
\alpha^*(f)(gh)(x) = \langle (\alpha \otimes \Delta) \circ \Delta(x),f \otimes g \otimes h \rangle.
\end{gather*}
So the Hom-associativity $\mu \circ (\mu \otimes \alpha^* - \alpha^* \otimes \mu) = 0$ follows from the Hom-coassociativity $(\Delta \otimes \alpha - \alpha \otimes \Delta) \circ \Delta = 0$. \\
Moreover, if $A$ has a counit $\epsilon
$ satisfying $(id \otimes \epsilon
) \circ \Delta = \alpha = (\epsilon
 \otimes id) \circ \Delta$ then for $f \in A^*$ and $x \in A$ we have
\begin{gather*}
(\epsilon
 f)(x) = \sum_{(x)} \epsilon
(x_1) f(x_2) = \sum_{(x)} f(\epsilon
(x_1) x_2) = f(\alpha(x)) = \alpha^*(f)(x) \\
\shortintertext{and}
(f \epsilon
)(x) = \sum_{(x)} f(x_1) \epsilon
(x_2) = \sum_{(x)} f(x_1 \epsilon
(x_2)) = f(\alpha(x)) = \alpha^*(f)(x),
\end{gather*}
which shows that $\epsilon
$ is the unit in $A^*$.
\end{proof}

The dual of a Hom-algebra $(A,\mu,\alpha)$ is not always a Hom-coalgebra, because the coproduct does not land in the good space: $\mu^* : A^* \to (A \otimes A)^* \supsetneq A^* \otimes A^*$. Nevertheless, it is the case if the Hom-algebra is finite dimensional, since $(A\otimes A)^* = A^* \otimes A^*$.

In the general case, for any Hom-algebra $A$, define
\begin{equation*}
A^\circ = \{f \in A^*, f(I)=0\ \text{for some cofinite ideal $I$ of $A$}\},
\end{equation*}
where a \emph{cofinite ideal} $I$ is an ideal $I \subset A$ such that $A/I$ is finite-dimensional. 
Recall that $I$ is 
an ideal of $A$ if for 
$x \in I$ and $y \in A$,  we have $x y\in I$, $y x\in I$ and $\alpha (x)\in I$.

$A^\circ$ is a subspace of $A^*$ since it is closed under multiplication by scalars and the sum of two elements of $A^\circ$ is again in $A^\circ$ since the intersection of two cofinite ideals is again a cofinite ideal. If $A$ is finite dimensional, of course $A^\circ = A^*$.
\begin{lemma}
Let $A$ and $B$ be two Hom-associative algebras and $f : A \to B$ be a Hom-algebra morphism. Then the dual map $f^* : B^* \to A^*$ satisfies $f^*(B^\circ)\subset A^\circ.$
\end{lemma}
\begin{proof}
Let $J$ be a cofinite ideal of $B$ and $p:B\to B/J$ be the canonical map. Set $\widetilde{f}=p\circ f:A\to B/J$.

Observe that $f^{-1}(J)$ is an ideal of $A$. Indeed, for $x\in A$ we have  $f(x f^{-1}(J))=f(x)f(f^{-1}(J))=f(x) J \subset J$. Therefore $x  f^{-1}(J)\subset  f^{-1}(J)$. Also $\alpha_A( f^{-1}(J))=f^{-1}(\alpha_B(J))\subset f^{-1}(J)$.

We have the following exact sequence 
$0\to f^{-1}(J)\xrightarrow{i}  A\xrightarrow{\widetilde{f}} B/J \rightarrow 0.
$ 
Define a map $f_\star:A/f^{-1}(J) \to B/J$ by $f_\star(x+f^{-1}(J))=f(x)$. It induces an isomorphism $A/f^{-1}(J)\to Im \widetilde{f}$. Hence $A/f^{-1}(J)$ is finite dimensional.

Likewise, we have  $f^\star (B^\circ)\subset A^\circ $. Indeed, let $b^\star \in B$ such that $ker (b^\star) \supset J$. Then $ker (f^\star (b^\star)) \supset f^{-1}(J)$, since 
$< f^\star (b^\star), f^{-1}(J)>=< b^\star, f(f^{-1}(J))>=< b^\star, J>=0.$
\end{proof}
Using this lemma one may prove, similarly to \cite[Lemma 6.0.1]{Sweedler} that  $A^\circ \otimes A^\circ = (A \otimes A)^\circ$ and  the dual $\mu^* : A^* \to (A \otimes A)^*$ of the multiplication $\mu : A \otimes A\rightarrow A$ satisfies $\mu^*(A^\circ)  \subset A^\circ \otimes A^\circ$. Indeed, for $f \in A^*,\ x,y \in A$, we have $\langle \mu^*(f),x \otimes y \rangle = \langle f,xy \rangle$. So if $I$ is a cofinite ideal such that $f(I) = 0$, then $I \otimes A + A \otimes I$ is a cofinite ideal of $A \otimes A$ which vanish on $\mu^*(f)$.


\begin{theorem}
Let $(A,\mu,\alpha)$ be a multiplicative Hom-associative algebra. Then its finite dual is provided with a structure of  Hom-coassociative coalgebra $(A^\circ,\Delta,\alpha^\circ)$, where  $\Delta = \mu^\circ = \mu^*|_{A^\circ}$. 
  Moreover,  the Hom-coassociative coalgebra  is counital  whenever  $A$ is unital, where  $\epsilon
 : A^\circ \to \K$  is defined by $\epsilon
(f) = f(1_A)$.
\end{theorem}

\begin{proof}
The coproduct $\Delta$ is defined from $A^\circ$ to $A^\circ \otimes A^\circ$ by
\begin{equation*}
\Delta(f)(x \otimes y) = \mu^*|_{A^\circ}(f)(x \otimes y) = \langle \mu(x \otimes y),f \rangle = f(xy), \quad x,y \in A.
\end{equation*}
For $f,g,h \in A^\circ$ and $x,y \in A$, we have
\begin{gather*}
(\Delta \circ \alpha^\circ) \circ \Delta(f)(x \otimes y \otimes z) = \langle \mu \circ (\mu \otimes \alpha) (x \otimes y \otimes z),f \rangle
\shortintertext{and}
(\alpha^\circ \circ \Delta) \circ \Delta(f)(x \otimes y \otimes z) = \langle \mu \circ (\alpha\otimes \mu) (x \otimes y \otimes z),f \rangle.
\end{gather*}
So the Hom-coassociativity $(\Delta \otimes \alpha^\circ - \alpha^\circ \otimes \Delta) \circ \Delta = 0$  follows from the Hom-associativity $\mu \circ (\mu \otimes \alpha - \alpha \otimes \mu) = 0$. \\
Moreover, if $A$ has a unit $\eta$ satisfying $\mu \circ (id \otimes \eta) = \alpha = \mu \circ (\eta \otimes id)$ then for $f \in A^\circ$ and $x \in A$ we have
\begin{gather*}
(\epsilon
 \otimes id) \circ \Delta(f)(x) = f(1_A\cdot x) = f(\alpha(x)) = \alpha^\circ(f)(x) \\
\intertext{and}
(id \otimes \epsilon
) \circ \Delta(f)(x) = f(x\cdot 1_A) = f(\alpha(x)) = \alpha^\circ(f)(x),
\end{gather*}
which shows that $\epsilon
 : A^\circ \to \K$, $f \mapsto f(1)$ is the counit in $A^\circ$.
\end{proof}
\begin{remark}
When $(A,\mu,\alpha)$ is a finite dimensional  Hom-associative algebra, the assumption of multiplicativity is not needed. The dual is provided with a structure of  Hom-coassociative coalgebra $(A^\star,\Delta^\star,\alpha^\star)$, see \cite{HomAlgHomCoalg}. 
\end{remark}

Now, we consider a structure combining Hom-associative algebras and Hom-coassociative coalgebras.
\begin{definition}
A \emph{Hom-bialgebra} is a tuple $\left( A,\mu ,\alpha,\eta ,\Delta ,\beta,\epsilon
 \right)$ where
\begin{enumerate}
\item $\left( A,\mu ,\alpha,\eta \right)$
is a Hom-associative algebra with a unit  $\eta$,
\item  $\left( A,\Delta ,\beta,\epsilon
 \right)$ is a Hom-coassociative coalgebra with a counit $\epsilon
$,
\item  the linear maps $\Delta $ and $\epsilon
 $ are compatible with the multiplication  $\mu$ and the unit $\eta$, that is for $x,y\in A$
\begin{eqnarray}
&& \Delta \left(1_A \right) = 1_A \otimes 1_A, \\
&& \Delta \left( \mu(x\otimes y)\right) = \Delta \left( x \right) \cdot \Delta \left( y\right)
=\sum_{(x)(y)}\mu(x_1 \otimes y_1) \otimes \mu( x_2 \otimes y_2), \\
&& \epsilon
 \left( 1_A\right) =1, \\
&& \epsilon
 \left( \mu(x\otimes y)\right)
=\epsilon
 \left(x\right) \epsilon
 \left( y\right), \\
&& \epsilon
\circ \alpha \left( x\right)
=\epsilon
 \left( x\right),
\end{eqnarray}
\end{enumerate}
where the dot "$\cdot$" denotes the
multiplication on tensor product.

If $\alpha = \beta$ the Hom-bialgebra is denoted $(A,\mu,\eta,\Delta,\epsilon
,\alpha)$.
\end{definition}

A \emph{Hom-bialgebra morphism} (resp. weak Hom-bialgebra morphism) is a morphism which is either a Hom-algebra and Hom-coalgebra morphism (resp. weak morphism).

By combining Theorems \ref{twist-ass_Lie} and  \ref{thmConstrHomCoalg}, we obtain the following  Proposition. 
\begin{proposition}
Let $(A,\mu,\eta,\Delta,\epsilon,\alpha)$ be a Hom-bialgebra  and $\beta : A\rightarrow A$ be a Hom-bialgebra morphism. Then $(A,\mu_\beta,\eta,\Delta_\beta,\epsilon,\beta\circ \alpha)$ is a Hom-bialgebra.

In particular, if $(A,\mu,\eta,\Delta,\epsilon)$ be a bialgebra  and $\beta : A\rightarrow A$ be a bialgebra morphism  then $(A,\mu_\beta,\eta,\Delta_\beta,\epsilon,\beta)$ is a Hom-bialgebra.

\end{proposition}
This construction method of Hom-bialgebra, starting with a given Hom-bialgebra or a bialgebra  and a morphism,  is called twisting principle.

Notice that the category of Hom-bialgebra is not closed under weak Hom-bialgebra morphisms.

Combining  previous observations, we show that a finite dual of a Hom-bialgebra is a Hom-bialgebra. 
\begin{proposition}
Let  $\left( A,\mu,\eta, \Delta,\epsilon ,\alpha \right) $ be a Hom-bialgebra. Then the finite dual $\left( A^\circ ,\mu^\circ,\eta^\circ,  \Delta^\circ,\epsilon^\circ, \alpha^\circ  \right)$ is a Hom-bialgebra as well.
\end{proposition}

Given a Hom-bialgebra $\left( A,\mu ,\alpha,\eta ,\Delta,\beta ,\epsilon
 \right)$, it is  shown in \cite{HomHopf,HomAlgHomCoalg} that the vector space $Hom \left( A,A \right)$ with the multiplication given
by the convolution product carries a structure of Hom-associative algebra.

\begin{proposition}
Let  $\left( A,\mu ,\alpha,\eta ,\Delta,\beta ,\epsilon
 \right) $
be a Hom-bialgebra. Then  $Hom \left( A,A \right)$ with
the multiplication given by the convolution product defined by
$$ f \star g=\mu \circ \left( f\otimes g \right) \circ\Delta $$
and the unit being $\eta \circ \epsilon
$,  is a unital Hom-associative
algebra with the homomorphism map defined by $\gamma (f)=\alpha
\circ f \circ \beta$.
\end{proposition}
Therefore, we have the following definitions:
\begin{enumerate}
\item An endomorphism $S$ of $A$ is said to be
an \emph{antipode} if it is  the inverse of the identity over $A$ for the Hom-associative algebra $Hom \left(A,A \right)$ with the multiplication given by the convolution product defined by
\[
f \star g= \mu \circ \left( f \otimes g \right) \Delta
\]
and the unit being $\eta \circ \epsilon
$.
\item A \emph{Hom-Hopf algebra} is a Hom-bialgebra with an antipode.
\end{enumerate}

\begin{remark}
We have the following properties :
\begin{itemize}
\item The antipode $S$   is unique,
\item   $S(1_A)=1_A $,
\item   $\epsilon
 \circ S=\epsilon
$.
\item  Let $x$ be a primitive element ($\Delta (x)=1_A\otimes x + x\otimes 1_A$), then
$\epsilon
(x)=0$.
\item  If $x$ and $y$ are two primitive elements in $A$. Then we have $\epsilon
 (x)=0$ and  the commutator
 $[x,y]=\mu (x\otimes
y) -\mu (y \otimes x)$ is also a primitive element.
\item The set of all primitive elements of $A$,
denoted by $Prim(A)$, has a structure of Hom-Lie algebra.
\end{itemize}
\end{remark}

\begin{example}
Let   $\K G$ be the group-algebra over the group $G$. As a vector
space, $\K G$ is generated by   $\{e_g : g \in G\}$. If
$\alpha:G\rightarrow G$ is a group homomorphism, then it can be
extended to an algebra endomorphism of  $\K G$ by setting
$$\alpha (\sum_{g\in G}{a_g e_g})=\sum_{g\in G}{a_g \alpha (e_g)}=\sum_{g\in G}{a_g e_{\alpha (g)}}.
$$
Consider the usual bialgebra structure on $\K G$ and $\alpha$  a
 bialgebra morphism.  Then,   we define over $\K G$ a
Hom-bialgebra $(\K G, \mu, \eta,\Delta,\epsilon, \alpha)$  by
setting:
$$\mu(e_g\otimes e_{g'})=\alpha (e_{g\cdot g'}),
$$
$$
\Delta \left( e_{g}\right) =\alpha(e_{g})\otimes \alpha(e_{g}).
$$
\end{example}

\begin{example}
We provide a one-parameter family of  twistings ($\lambda\in \K $), which deforms Sweedler Hopf algebra to Hom-bialgebras. \\
We consider the basis $ \{e_1 =1,\ e_2 =c,\ e_3=x,\ e_4=cx \} $.

 The multiplication, with respect to the basis, is 
\begin{eqnarray*}
&& \mu(e_1, e_1)=e_1,\  \mu(e_1, e_2)=e_2, \ \mu(e_1, e_3)=\lambda e_3,\ \mu(e_1, e_4)=\lambda e_4, \\
&& \mu(e_2, e_1)=e_2,\  \mu(e_2, e_2)=e_1, \ \mu(e_2, e_3)=\lambda e_4,\ \mu(e_2, e_4)=\lambda e_3, \\
&& \mu(e_3, e_1)=\lambda e_3,\  \mu(e_3, e_2)=-\lambda e_4, \ \mu(e_4, e_1)=\lambda e_4,\ \mu(e_4, e_2)=-\lambda e_3. 
\end{eqnarray*} 

The comultiplication is defined by
\begin{eqnarray*}
&& \Delta (e_1 )=e_1\otimes e_1,\ \Delta (e_2 )=e_2\otimes e_2,\\
&& \Delta (e_3 )=\lambda ( e_2\otimes e_3+e_3\otimes e_1),\  \Delta (e_4 )=\lambda (e_1\otimes e_4+e_4\otimes e_2).
\end{eqnarray*} 

The counit is given by $\epsilon
 (e_1)=\epsilon
 (e_2)=1,\  \epsilon
 (e_3)=\epsilon
 (e_4)=0.$ 

The twist map is defined by  $\alpha (e_1) =e_1,\  \alpha (e_2) =e_2, \  \alpha (e_3) =\lambda e_3,\ \alpha (e_4) =\lambda e_4.$
 \end{example}

\begin{example} Consider the  polynomial algebra  $A=\K [(X_{ij})]$ in variables $(X_{i j})_{i,j=1,\cdots,n}$. It carries a  structure of bialgebra with the comultiplication  defined by
$ \delta (X_{i j})=\sum_{k=1}^{n}{X_{i k}\otimes X_{k j}}$ and
$\delta (1)=1\otimes 1$. Let $\alpha$ be a bialgebra
morphism, it is defined by $n^2$ polynomials $\alpha (X_{i j})$. We
define a  Hom-bialgebra $(A, \mu, \alpha,\Delta,\alpha)$
by
\begin{align*}
\mu(f\otimes g)&=f(\alpha (X_{1 1}),\cdots,\alpha (X_{n n}))g(\alpha (X_{1 1}),\cdots,\alpha (X_{n n})),\\
 \Delta (X_{i j})&=\sum_{k=1}^{n}{\alpha (X_{i k})\otimes \alpha (X_{k j})},\\
 \Delta (1)&=\alpha (1)\otimes \alpha (1).
\end{align*}
\end{example}

\begin{example}[Universal enveloping Hom-algebra]
A Hom-Lie algebra is a Hom-algebra $(A,[~,~],\alpha)$ satisfying a twisted Jacobi condition $\circlearrowleft_{x,y,z}{[\alpha(x),[y,z]]},$ where the bracket denotes the product and  $\circlearrowleft_{x,y,z}$ denotes the cyclic sum  on $x,y,z$.

Given a multiplicative Hom-associative algebra $A = (A,\mu,\alpha)$, one can associate to it a multiplicative Hom-Lie algebra $HLie(A) = (A,[~,~],\alpha)$ with the same underlying module $(A,\alpha)$ and the bracket $[~,~] = \mu - \mu^{op}$. 
This construction gives a functor $HLie$ from multiplicative Hom-associative algebras to multiplicative Hom-Lie algebras \cite{MakSil-struct}. In \cite{Yau:EnvLieAlg}, Yau constructed the left adjoint $U_{HLie}$ of $HLie$. He also made some minor modifications in \cite{Yau:comodule} to take into account the unital case.

The functor $U_{HLie}$ is defined as
\begin{equation}
U_{HLie}(A) = F_{HNAs}(A)/I^\infty \qquad \text{with} \quad F_{HNAs}(A) = \oplus_{n \geqslant 1} \oplus_{\tau \in T_n^{wt}} A_{\tau}^{\otimes n}
\end{equation}
for a multiplicative Hom-Lie algebra $(A,[~,~],\alpha)$. Here $T_n^{wt}$ is the set of weighted $n$-trees encoding the multiplication of elements (by trees) and twisting by $\alpha$ (by weights), $A_{\tau}^{\otimes n}$ is a copy of $A^{\otimes n}$ and $I^\infty$ is a certain submodule of relations build in such a way that the quotient is Hom-associative.

Moreover, the comultiplication $\Delta : U_{HLie}(A) \to U_{HLie}(A)~\otimes~U_{HLie}(A)$ defined by $\Delta(x) = \alpha(x) \otimes 1 + 1 \otimes \alpha(x)$ equips the multiplicative Hom-associative algebra $U_{HLie}(A)$ with a structure of Hom-bialgebra.
\end{example}

\section{Hom-quasi-bialgebras (HQ-bialgebras)}\label{Quasi}
The aim of this section is to extend the concept of quasi-bialgebra  to Hom-setting. This generalized structure  is called Hom-quasi-bialgebra  which we write for  shortness HQ-bialgebra.
  Moreover, we introduce  quasi-triangular HQ-bialgebra. We provide a construction deforming a HQ-bialgebra with HQ-bialgebra morphism to a new HQ-bialgebra.

 \begin{definition}  \label{Quasibi} A {\rm Hom-quasi-bialgebra} (HQ-bialgebra for short)  is a
 
 tuple $(A, \mu, \eta, \Delta, \epsilon
, \alpha, \Phi)$, where $\Phi$ is an invertible element in $A^{\otimes 3}$ satisfying $\alpha^{\otimes 3}(\Phi)=\Phi$, and  such that  
 
 \begin{enumerate}
\item  
the quadruple $(A, \mu,  \alpha, \eta)$ is a unital Hom-associative algebra,

\item
the map $\Delta: A\rightarrow A\otimes A$ is a Hom-algebra morphism, that is  $\; \Delta\circ \alpha=\alpha^{\otimes 2} \circ\Delta$ and $\Delta( xy)=\Delta(x)\Delta(y)$,

\item
the map $\epsilon :  \K \rightarrow A$ is a Hom-algebra morphism, that is  $\; \epsilon (\alpha(x))=\epsilon
(x)$ and $\epsilon ( xy)=\epsilon (x)\epsilon (y)$, 
\item and  the following four identities  are satisfied

\end{enumerate}
\begin{eqnarray}
 (\alpha \otimes \Delta)(\Delta(x))\Phi&=& \Phi(\Delta \otimes \alpha)(\Delta(x)) , \label{HQ1}\\
\label{HQ2}(\epsilon
 \otimes id)(\Delta(x))&=& \alpha(x) \;=\; (id \otimes \epsilon
)(\Delta(x)),\\
\label{HQ3}(\alpha \otimes \alpha \otimes \Delta)(\Phi)(\Delta \otimes \alpha \otimes \alpha)(\Phi)&=&\Phi_{(234)}(\alpha \otimes \Delta \otimes \alpha)(\Phi)\Phi_{(123)},\\
\label{HQ4}(id \otimes \epsilon
 \otimes id)(\Phi)&=& 1_A \otimes 1_A,
\end{eqnarray}
where   $\Phi_{(123)}=\Phi \otimes 1_A$  and  $\Phi_{(234)}=1_A \otimes \Phi$.

 \end{definition}
 \noindent
\begin{remark}\   
\begin{itemize}
\item  Equation \eqref{HQ1} of Definition  \ref{Quasibi}  is equivalent to 
 \begin{eqnarray*}
 R_{\Phi}\circ(\alpha \otimes id ) \circ (id \otimes \Delta) \circ \Delta
&=&    L_{\Phi} \circ     (id \otimes \alpha)\circ (\Delta \otimes id)\circ \Delta, 
\end{eqnarray*}
where $R_{\Phi}$ (resp. $L_{\Phi}$) is right (resp. left) multiplication by $\Phi$.
\item The parenthesis in the right hand side  of identity  \eqref{HQ3} are not needed since $\alpha^{\otimes 3} ( \Phi  )=   \Phi$.
\item The conditions \eqref{HQ3},\eqref{HQ4} may be interpreted as $\Phi$ has to be a $3$-cocycle.
\item If $A$ is a HQ-bialgebra then $A^{op},\ A^{cop}$ and $A^{op,cop}$ are also HQ-bialgebras. The structures are obtained by setting respectively  $\Phi_{op}=\Phi^{-1}$,  $\Phi_{cop}=\Phi^{-1}_{321}$ and $\Phi_{op,cop}=\Phi_{321}$. If $\Phi=\sum \Phi^{(1)} \otimes \Phi^{(2)}\otimes \Phi^{(3)}$, then  $\Phi_{321}=\sum \Phi^{(3)} \otimes \Phi^{(2)}\otimes \Phi^{(1)}$ and so on.

\end{itemize}
\end{remark} 
  \begin{definition}  
 
 Let $(A, \mu, \eta, \Delta, \epsilon
, \alpha, \Phi)$ and $(A', \mu', \eta', \Delta', \epsilon
', \alpha', \Phi')$ be HQ-bialgebras.   A morphism $f: A \rightarrow A'$ is called HQ-bialgebra morphism if the following conditions are satisfied
\begin{eqnarray}
&f\circ  \mu= \mu' \circ (f\otimes f), \quad &f ( 1_A)=  1_{A'} ,\\
&(f\otimes f)  \circ  \Delta= \Delta' \circ f, \quad 
&\epsilon
 ' \circ f= \epsilon
,\\
&f\circ \alpha=\alpha'\circ f, & f ^{\otimes 3}(\Phi) = \Phi' .     
\end{eqnarray} 
 
 \end{definition}
 
 \begin{remark}
The dual HQ-bialgebra could be defined in a natural way.  Similarly to the case of quasi-bialgebra \cite{Majid}, it turns out  that the dual HQ-bialgebra is Hom-associative only up to conjuagtion in a suitable convolution algebra by a $3$-cocycle $\Phi$. The axioms could be found by dualization.
 \end{remark}
 
In the following, we show how to construct HQ-bialgebras starting from a given HQ-bialgebra and a HQ-bialgebra morphism. In particular, a quasi-bialgebra and a quasi-bialgebra morphism lead to a HQ-bialgebra. This extends twisting principle to HQ-bialgebras.
\begin{theorem}\label{Thm:TwistHQbialgra}

Let $(A, \mu, \eta, \Delta, \epsilon
, \alpha, \Phi)$ be a HQ-bialgebra and $\beta : A \rightarrow A$ be a HQ-bialgebra morphism.\\ Then $(A, \mu_{\beta}=\beta\circ \mu, \eta, \Delta_{\beta}=\Delta\circ \beta, \epsilon, \beta\circ \alpha, \Phi)$ is a HQ-bialgebra.

In particular, if $(A, \mu, \eta, \Delta, \epsilon,  \Phi)$ is a quasi-bialgebra and $\beta : A \rightarrow A$ is a quasi-bialgebra morphism, 
 then $(A, \beta\circ \mu, \eta,\Delta\circ \beta, \epsilon, \beta, \Phi)$ is a HQ-bialgebra.
\end{theorem}
\begin{proof}  
The proof consists in checking   axioms of Definition \ref{Quasibi}.  Since $\beta^{\otimes 3}(\Phi)=\Phi$, $\beta^{\otimes 3}(\Phi^{-1})=\Phi^{-1}$ and $\beta \alpha=\alpha \beta$, the left hand side of the identity \eqref{HQ1} may be written 
\begin{eqnarray*}
((\beta\alpha \otimes \Delta_{\beta})(\Delta_\beta(x)))\Phi&=& ((\beta\alpha \otimes \Delta\beta )(\Delta \beta (x)))\Phi \\
&=& (\beta^{\otimes 3}(\alpha \otimes \Delta ) \beta ^{\otimes 2}\Delta(x)))\Phi\\
&=&\beta^{\otimes 3} \beta^{\otimes 3}(((\alpha \otimes \Delta ) \Delta(x))\Phi).
\end{eqnarray*}
On the other hand, the right hand side may be written as  
\begin{eqnarray*}
\Phi(( \Delta_{\beta} \otimes \beta\alpha )(\Delta _{\beta}(x)))&=&\Phi(( \Delta\beta \otimes \beta\alpha )(\Delta \beta(x)))\\
&=&\Phi(\beta^{\otimes 3}(\Delta\otimes \alpha) (\beta ^{\otimes 2}\Delta(x))) \\
&=&\beta^{\otimes 3}(\Phi (\Delta \otimes\alpha) (\beta ^{\otimes 2}\Delta(x))) \\
&=&\beta^{\otimes 3}(\Phi \beta^{\otimes 3}(\Delta \otimes\alpha ) (\Delta(x))) \\
&=&\beta^{\otimes 3}  \beta^{\otimes 3} ( \Phi (\Delta \otimes \alpha) (\Delta(x)) ).
\end{eqnarray*}
Since identity \eqref{HQ1}  is satisfied by  the initial HQ-bialgebra,  we are done.\\
Identity \eqref{HQ2}  is straightforward. For  identity \eqref{HQ3}, we have
\begin{eqnarray*}
LHS=(\beta\alpha \otimes \beta\alpha \otimes \Delta_{ \beta})(\Phi)(\Delta_{ \beta} \otimes\beta \alpha \otimes\beta \alpha)(\Phi)\\
= \beta^{\otimes 4}((\alpha \otimes\alpha \otimes \Delta)(\Phi)(\Delta \otimes\alpha\otimes \alpha)(\Phi)),
\end{eqnarray*}
and 
\begin{eqnarray*}
RHS=\Phi_{(234)}(\beta\alpha \otimes \Delta_{\beta} \otimes \beta\alpha)(\Phi)\Phi_{(123)}
= \beta^{\otimes 4} ((\Phi_{(234)})  \;  ((\alpha \otimes \Delta \otimes \alpha)(\Phi))) \Phi_{(123)} ).
\end{eqnarray*}
It follows $LHS=RHS$.  Identity \eqref{HQ4} is checked exactly the same way.
\end{proof}

\begin{corollary}
A multiplicative HQ-bialgebra  $(A, \mu, \eta, \Delta, \epsilon
, \alpha, \Phi)$,  yields infinitely many  HQ-bialgebras $(A,\alpha^n\circ \mu, \eta, \Delta\circ\alpha^n, \epsilon
, \alpha^{n+1}, \Phi)$.
\end{corollary}

\begin{example}[Twisted quantum double of $H(2)$]
Let $H(2)$ be the $2$-dimensional bialgebra $\K C_2$, the group algebra associated to cyclic group of order $2$. The quantum double of $H(2)$ is the unital algebra generated by $X$ and $Y$ with relations $X^2=1,\ Y^2=X$ and $XY=YX$ (we can view it as the algebra generated by $Y$ such that $Y^4=1$). The comultiplication of $D(H(2))$ is given by the formulas
$\Delta(X)=X\otimes X$ and $\Delta(Y)=-\frac{1}{2}(Y\otimes Y+XY\otimes Y+Y\otimes XY-XY\otimes XY).$ The counit is defined as $\epsilon(X)=1,\ \epsilon(Y)=-1$. It carries a structure of quasi-bialgebra with $\Phi=1-2P\otimes P\otimes P,$ where $P=\frac{1}{2}(1-X)$, (see for example \cite[Proposition 3.10]{BulacuCaneTorr}).

We assume in this example that $\K=\mathbb{C}$. The quasi-algebra morphisms of $D(H(2))$ are defined as 
\begin{enumerate}
\item $\alpha_1(X)=X,\ \alpha_1(Y)=XY,$
\item $\alpha_2(X)=X,\ \alpha_2(Y)=\xi 1+\bar{\xi}X,$ where $\xi=e^{\frac{3\pi}{4}\imath}$ and $\bar{\xi}$ the conjugate of $\xi$.
\item $\alpha_3(X)=X,\ \alpha_3(Y)= \bar{\xi} 1+\xi X$.
\end{enumerate}
Theorem \ref{Thm:TwistHQbialgra} leads to the following two examples, corresponding to $\alpha_1$ and $\alpha_2$, of HQ-bilagebras:

Let  $\{e_1,e_2,e_3,e_4\} $ be the basis of the $4$-dimensional vector space and $e_1$ being the unit. The counit for both  is defined as $\epsilon(e_1)=\epsilon(e_2)=1$ and $\epsilon(e_3)=\epsilon(e_4)=-1$ and $\Phi=1-2P\otimes P\otimes P$ where $P=\frac{1}{2}(e_1-e_2)$. The multiplications and comultiplications are defined as follows
\begin{itemize}
\item 
For the map $\alpha_1$ defined by
$$\alpha_1(e_1)=e_1,\ \alpha_1(e_2)=e_2,\ \alpha_1(e_3)=e_4, \alpha_1(e_4)=e_3,$$
\noindent the multiplication   is
\begin{eqnarray*}
e_1\cdot_1 e_1=e_1,\ e_1\cdot_1 e_2=e_2,\ e_1\cdot_1 e_3=e_4,\  e_1\cdot_1 e_4=e_3,\\
e_2\cdot_1 e_1=e_2,\ e_2\cdot_1 e_2=e_1,\ e_2\cdot_1 e_3=e_3,\  e_2\cdot_1 e_4=e_4,\\
e_3\cdot_1 e_1=e_4,\ e_3\cdot_1 e_2=e_3,\ e_3\cdot_1 e_3=e_2,\  e_3\cdot_1 e_4=e_1,\\
e_4\cdot_1 e_1=e_3,\ e_4\cdot_1 e_2=e_4,\ e_4\cdot_1 e_3=e_1,\  e_4\cdot_1 e_4=e_2,
\end{eqnarray*}
\noindent and the comultiplication is 
\begin{eqnarray*}
&& \Delta(e_1)=e_1\otimes e_1,  \Delta(e_2)=e_2\otimes e_2,\\
 && \Delta(e_3)=-\frac{1}{2}(e_4\otimes e_4+e_3\otimes e_4+e_4\otimes e_3-e_3\otimes e_3),\\
 && \Delta(e_4)=-\frac{1}{2}(e_3\otimes e_3+e_4\otimes e_3+e_3\otimes e_4-e_4\otimes e_4).
\end{eqnarray*}
\item For the  map $\alpha_2$ defined by
$$\alpha_2(e_1)=e_1,\ \alpha_2(e_2)=e_2,\ \alpha_2(e_3)=\xi e_1+\bar{\xi}e_2, \alpha_2(e_4)= \bar{\xi}e_1+\xi e_2,$$
\noindent the multiplication  is
\begin{eqnarray*}
e_1 \cdot_2 e_1=e_1,\ e_1 \cdot_2 e_2=e_2,\ e_1 \cdot_2 e_3=\xi e_1+\bar{\xi}e_2,\  e_1 \cdot_2 e_4=\bar{\xi}e_1+\xi e_2,\\
e_2 \cdot_2 e_1=e_2,\ e_2 \cdot_2 e_2=e_1,\ e_2 \cdot_2 e_3=\bar{\xi}e_1+\xi e_2,\  e_2 \cdot_2 e_4=\xi e_1+\bar{\xi}e_2,\\
e_3 \cdot_2 e_1=\xi e_1+\bar{\xi}e_2, \ e_3 \cdot_2 e_2=\bar{\xi}e_1+\xi e_2,\ e_3 \cdot_2 e_3=e_2,\  e_3 \cdot_2 e_4=e_1,\\
e_4 \cdot_2 e_1=\bar{\xi}e_1+\xi e_2,\ e_4 \cdot_2 e_2=\xi e_1+\bar{\xi}e_2,\ e_4 \cdot_2 e_3=e_1,\  e_4 \cdot_2 e_4=e_2,
\end{eqnarray*}
\noindent and the comultiplication is 
\begin{eqnarray*}
&& \Delta(e_1)=e_1\otimes e_1,  \Delta(e_2)=e_2\otimes e_2,\\
&&  \Delta(e_3)=-e_1\otimes e_1+e_1\otimes e_2+e_2\otimes e_1-e_2\otimes e_2,\\
&& \Delta(e_4)=-e_1\otimes e_1+e_1\otimes e_2+e_2\otimes e_1-e_2\otimes e_2.\\
\end{eqnarray*}
\end{itemize}
\end{example}

In the following we   define quasi-triangular HQ-bialgebras and provide a construction using a twisting principle, that is we obtain a new quasi-triangular HQ-bialgebra from a given quasi-triangular HQ-bialgebra and a morphism. Quasi-triangular bialgebra was introduced in \cite{DRIN2} and extended to Hom-bialgebra in \cite{Yau:YangBaxter,Yau:HQuantumGroup1}.

\begin{definition}   

A {\rm quasi-triangular HQ-bialgebra}  is a tuple $(A, \mu, \eta,\Delta, \epsilon
,  \alpha, \Phi, R)$ in which  $(A, \mu, \eta, \Delta, \epsilon
,   \alpha, \Phi)$  is a HQ-bialgebra   and $R$ is an invertible element in  $A \otimes A$,  satisfying  $\alpha^{\otimes 2} (R)=R$, such that  for all $x \in A$,
\begin{eqnarray}
\label{QT1} \Delta^{op}(x)&=& R\; \Delta(x) R^{-1},\\
\label{QT2} (\alpha \otimes \Delta)(R)&=&\Phi_{231}^{-1} R_{13}\Phi_{213}R_{12} \Phi^{-1},\\
\label{QT3} (\Delta \otimes \alpha)(R)&=&\Phi_{312}R_{13}\Phi_{123}^{-1}R_{23} \Phi,
\end{eqnarray}
where $R_{13}=\tau_{23} (R \otimes 1), R_{23}=1 \otimes R$ and $R_{12}=R \otimes 1$. 
 \end{definition}
 Notice that we do not need to fix the bracket in the right hand sides of identities \eqref{QT1},\eqref{QT2},\eqref{QT3}  since $\alpha^{\otimes 3} (\Phi)=\Phi$ and $\alpha^{\otimes 2} (R)=R$.
 
 A quasi-triangular QH-bialgebra morphism is a QH-bialgebra morphism which conserve the element $R$. 
 
Likewise a  quasi-triangular HQ-bialgebra and a quasi-triangular HQ-bialgebra morphism give rise to a new quasi-triangular HQ-bialgebra.
 
 \begin{theorem}
 Let $(A, \mu, \eta,\Delta, \epsilon
,  \alpha, \Phi, R)$ be a quasi-triangular HQ-bialgebra and $\beta: A \rightarrow A$ be a  quasi-triangular HQ-bialgebra morphism.\\
 Then $(A, \mu_\beta=\beta\circ \mu, \eta,\Delta_\beta=\Delta\circ\beta, \epsilon
, \beta\circ \alpha, \Phi, R)$  is a quasi-triangular HQ-bialgebra.
 \end{theorem}

\begin{proof}  
It follows from Theorem \ref{Thm:TwistHQbialgra} that $(A, \mu_\beta, \eta,\Delta_\beta, \epsilon
,  \beta\circ \alpha, \Phi)$ is a HQ-bialgebra.  The proof for the   identities \eqref{QT1},\eqref{QT2},\eqref{QT3}  is obvious. Indeed for example 
$$( \beta\alpha \otimes \Delta_\beta)(R)=\beta ^{\otimes 3} (\alpha \otimes \Delta)(R),$$
and 
$$\Phi_{231}^{-1} R_{13}\Phi_{213}R_{12} \Phi_{123}^{-1}=\beta ^{\otimes 3}(\Phi_{231}^{-1} R_{13}\Phi_{213}R_{12} \Phi_{123}^{-1}).
$$
\end{proof}
In particular,  a quasi-triangular quasi-bialgebra and a morphism give rise to  a quasi-triangular HQ-bialgebra.\\

\section{Gauge Transformation and Drinfeld's Twist Construction for HQ-bialgebras}\label{DrinfeldTwist}
The aim of this section is to extend the Drinfeld's twist construction to HQ-bialgebras. We prove the main theorem stating that gauge transformations give rise to new HQ-bilagebra structure.

In this section we assume that HQ-bialgebras $(A, \mu, \eta, \Delta, \epsilon
, \alpha, \Phi)$ are multiplicative and  $\alpha (1_A)=1_A$. One can easily see that for all $x\in A$, we have $\alpha(x) x=x\alpha(x).$
  \begin{definition}  
 
 Let $(A, \mu, \eta, \Delta, \epsilon
, \alpha, \Phi)$ be a HQ-bialgebra.  A gauge transformation on $A$ is an invertible element $F$ of $A^{\otimes 2}$ such that $$\alpha^{\otimes 2}(F)=F\; \mbox{ and}\; \; (\epsilon
 \otimes id)(F)=(id \otimes \epsilon
)(F)=1_A .$$
 \end{definition}
 \noindent We set
 $$\Delta_F(x)=(F  \Delta(x))  F^{-1}, \; \forall x \in A$$ and $$\Phi_F=F_{23}( (\alpha \otimes \Delta)(F) (\Phi ((\Delta \otimes \alpha)(F^{-1})F_{12}^{-1}))),$$
  where $F_{23}=1_A \otimes F$ and $F_{12}=F \otimes 1_A$.
\\  
  Observe that $\Delta_F(x)=(F  \Delta(x))  F^{-1}=F(  \Delta(x)  F^{-1}).$

\begin{lemma}\label{lem4.2}
The  following identities hold for all $x,y\in A$
\begin{eqnarray} (\alpha^3 \otimes \Delta_F) (x \otimes y)=  F_{23}( (\alpha \otimes \Delta)(x \otimes y)) F_{23}^{-1}, \\
 ( \Delta_F \otimes \alpha^3) (x \otimes y)=  F_{12}(  \Delta  \otimes \alpha)(x \otimes y)) F_{12}^{-1}.
 \end{eqnarray}

\end{lemma}\label{lem4.3}
\begin{proof}  
We have 
{\small{
\begin{eqnarray*}
 (\alpha^3 \otimes \Delta_F) (x \otimes y)&=&\alpha^3( x) \otimes \Delta_F (y)=\alpha^3( x) \otimes ( (F \cdot \Delta(y)) \cdot F^{-1}) 
 \\ \  &=&( (1_A \otimes F) \cdot (\alpha(x) \otimes \Delta(y))) \cdot (1_A \otimes F^{-1})\\
 \ &=& F_{23}( (\alpha \otimes \Delta)(x \otimes y)) F_{23}^{-1}. 
 \end{eqnarray*}
}}
 Th proof for the second identity is  similar.
 \end{proof}

  \begin{lemma}\label{lem4.5}
  The maps $\alpha \otimes \Delta$ and  $ \Delta_F$ are algebra morphisms.
    \end{lemma}

\begin{proof}  
The first map is a tensor product of two algebra morphisms. For $ \Delta_F$, using Hom-associativity, the fact that $\alpha^{\otimes 2}(F)=F$ and unitality, 
we obtain  
{\small{
 \begin{eqnarray*}
\Delta_F(x)  
 \Delta_F(y)&=&  ((F  
 \Delta(x))  
 F^{-1})  
 ( (F  
 \Delta(y))  
 F^{-1})\\
&=&( ( (F  
 \Delta( \alpha^{-1}(x)))  
 F^{-1})  
  (F  
 \Delta(y)))  
 F^{-1} \\
&=& ((F  
 \Delta(x))  
 (F^{-1}  
  (F  
 \Delta( \alpha^{-1}(y))))  
 F^{-1} \\
&=& ((F  
 \Delta(x))  
 ((1_A\otimes 1_A  )
 \Delta( y)))  
 F^{-1} 
= ((F  
 \Delta(x))  
 \Delta(\alpha(y)))   
 F^{-1} \\
&=& (F  
 (\Delta(x)  
 \Delta( y)) ) 
 F^{-1} 
 = (F  
 \Delta(x y))  
 F^{-1} \\
&=& \Delta_F(x   y).
 \end{eqnarray*}
 }} 
\end{proof} 

  \begin{theorem}\label{ThmConstructGauge}
Let  $(A, \mu, \eta, \Delta, \epsilon, \alpha, \Phi)$ be a HQ-bialgebra and   $F$ be a gauge transformation on $A$.
 Then $(A, \mu, \eta, \Delta_F, \epsilon, \alpha^3, \Phi_F)$ is a HQ-bialgebra.
\end{theorem}

\begin{proof}
 First, we check that
$(\alpha^3 \otimes \Delta_F)(\Delta_F(x)  
 \Phi_F= \Phi_F  
 (\Delta_F \otimes \alpha^3)(\Delta_F(x))$.  

Using Lemma \ref{lem4.2},  Hom-associativity  and the fact that $\alpha^{\otimes 2}(F)=F$, we get
{\small{
\begin{eqnarray*}
&& (\alpha^3 \otimes \Delta_F) (\Delta_F(x))  
 \Phi_F\\
&&=   (
 (
F_{23}  
 (\alpha \otimes \Delta )(( F  
 \Delta(x))  
 F^{-1} )
)  
 {F_{23}}^{-1}  )
  (
F_{23}((\alpha \otimes \Delta)(F) (
\Phi  
  (
(\Delta \otimes \alpha)(F^{-1})  
 F_{12}^{-1} )
 )
 )
 )
 \\
&&=   (
F_{23}  
 (\alpha \otimes \Delta )(( F  
 \Delta(\alpha(x)))  
 F^{-1} )
)  
 ({F_{23}}^{-1}  
  (
F_{23}((\alpha \otimes \Delta)(F) (
\Phi  
  (
(\Delta \otimes \alpha)(F^{-1})  
 F_{12}^{-1} )
 )
 )
 )).
 \end{eqnarray*}}}
 Since $F_{23}^{-1}  
F_{23}=1_A\otimes 1_A\otimes1_A,$  it becomes
 {\small{
\begin{eqnarray*}
&& (\alpha^3 \otimes \Delta_F) (\Delta_F(x))  
 \Phi_F\\
&&=   (
F_{23}  
 (\alpha \otimes \Delta )(( F  
 \Delta(\alpha(x)))  
 F^{-1} )
)  
 ((\alpha \otimes \Delta)(F) (
\Phi  
  (
(\Delta \otimes \alpha)(F^{-1})  
 F_{12}^{-1} )
 )
 )
 \\
&&=   (
 (
F_{23}  
 (\alpha \otimes \Delta )(( F  
 \Delta(x))  
 F^{-1} )
)  
 (\alpha \otimes \Delta)(F)) (
\Phi  
  (
(\Delta \otimes \alpha)(F^{-1})  
 F_{12}^{-1} )
 )
 \\
&&=   (
F_{23}  
  (
 (\alpha \otimes \Delta )(( F  
 \Delta(x))  
 F^{-1} ) 
 (\alpha \otimes \Delta)(F) )
 ) 
  (
\Phi  
  (
(\Delta \otimes \alpha)(F^{-1})  
 F_{12}^{-1} )
 ).
 \end{eqnarray*}
 }}

 Using that $\alpha \otimes \Delta$ is an algebra morphism, Hom-associativity and that $F^{-1}F=1$, the result above can be rewritten 
 {\small{
 \begin{align*}
 &(\alpha^3 \otimes \Delta_F) (\Delta_F(x))  
 \Phi_F\\
&=   (
F_{23}  
  (
 (
 ((\alpha \otimes \Delta )( F)(\alpha \otimes \Delta )  
 \Delta(x)))  
 (\alpha \otimes \Delta ) (F^{-1} )
)  
 (\alpha \otimes \Delta)(F) )
 ) 
  (
\Phi  
  (
(\Delta \otimes \alpha)(F^{-1})  
 F_{12}^{-1} )
 )
 \\
&=  (
F_{23}  
( (\alpha \otimes \Delta )(F)  
 (\alpha \otimes \Delta) \Delta(\alpha^2(x)))) 
  (
\Phi  
  (
(\Delta \otimes \alpha)(F^{-1})  
 F_{12}^{-1} )
 )\\
&=  F_{23}  (
 ( (\alpha \otimes \Delta )(F)  
 (\alpha \otimes \Delta) \Delta(\alpha^2(x)))  
  (
\Phi  
  (
(\Delta \otimes \alpha)(F^{-1})  
 F_{12}^{-1} )
 )
 )
 )
 \\
&=  F_{23}  (
 (
 ( (\alpha \otimes \Delta )(F)  
 (\alpha \otimes \Delta) \Delta(\alpha(x)))  
 \Phi )  
  (
(\Delta \otimes \alpha)(F^{-1})  
 F_{12}^{-1} )
 )
 \\
&=  F_{23}  (
 ( (\alpha \otimes \Delta )(F)  
  (
 (\alpha \otimes \Delta) \Delta(\alpha(x))  
 \Phi )
 )  
  (
(\Delta \otimes \alpha)(F^{-1})  
 F_{12}^{-1} )
 ).
 \end{align*}
 }}
 
We apply identity  \eqref{HQ1} for the HQ-bialgebra  $A$, it follows 
\begin{align*}
 (\alpha^3 \otimes \Delta_F) (\Delta_F(x))  
 \Phi_F
&=&  F_{23}  (
 (
(\alpha \otimes \Delta )(F)  
  (
 \Phi  
 (\Delta\otimes \alpha) \Delta(\alpha(x)))  )
  (
(\Delta \otimes \alpha)(F^{-1})  
 F_{12}^{-1} )
 )\\
 \ &=&F_{23}  (
(\alpha \otimes \Delta )(F)  
  ((
 \Phi  
 (\Delta\otimes \alpha) \Delta(\alpha(x)))  
  (
(\Delta \otimes \alpha)(F^{-1})  
 F_{12}^{-1} )
 )).
 \end{align*}
 
 We insert $(\Delta \otimes \alpha)(F^{-1}F)$ and using Hom-associativity we obtain
 
  {\small{
 \begin{align*}
& (\alpha^3 \otimes \Delta_F) (\Delta_F(x))  
 \Phi_F\\
&=  F_{23}  (
 (\alpha \otimes \Delta )(F)  
  (
 (
 (
 \Phi  
  (
 (\Delta \otimes \alpha)(F^{-1})   
 (\Delta \otimes \alpha)(F) ))  
 (\Delta \otimes \alpha) \Delta(\alpha(x)))  
  (
(\Delta \otimes \alpha)(F^{-1})  
 F_{12}^{-1} )
 )
 )
 \\
&=  F_{23}  (
 (\alpha \otimes \Delta )(F)  
  (
 (
 \Phi  
  (
 (\Delta \otimes \alpha)(F^{-1})   
 (\Delta \otimes \alpha)(F) ))
 ((\Delta \otimes \alpha) \Delta(\alpha(x))
  (
(\Delta \otimes \alpha)(F^{-1})  
 F_{12}^{-1} )
 )
 )
 )
 \\
&=  F_{23}  (
 (\alpha \otimes \Delta )(F)  
  (
 \Phi  
  ((
 (\Delta \otimes \alpha)(F^{-1})   
 (\Delta \otimes \alpha)(F)  )
  (
(\Delta \otimes \alpha) \Delta(x))    
  (
(\Delta \otimes \alpha)(F^{-1})  
 F_{12}^{-1} )
 )
 )
 )
 )
 \\
&=  (
 F_{23}  
 (\alpha \otimes \Delta )(F) ) 
  (
 \Phi  
  (
 (
 (\Delta \otimes \alpha)(F^{-1})   
 (\Delta \otimes \alpha)(F))  
  (
(\Delta \otimes \alpha) \Delta(\alpha(x)))  ) 
  (
(\Delta \otimes \alpha)(F^{-1})  
 F_{12}^{-1} )
 )
 )
 )
 \\
&=  ((
 F_{23}  
 (\alpha \otimes \Delta )(F) ) 
 \Phi)  
  (
 (
 (\Delta \otimes \alpha)(F^{-1})   
 (\Delta \otimes \alpha)(F))  
  (
(\Delta \otimes \alpha) \Delta(\alpha^2(x)))   
  (
(\Delta \otimes \alpha)(F^{-1})  
 F_{12}^{-1} )
 )
 ).
 \end{align*}
 }}

 We use again Hom-associativity and insert $F_{12}^{-1}F_{12}$. Then
 {\small{
 \begin{eqnarray*}
&& (\alpha^3 \otimes \Delta_F) (\Delta_F(x))  
 \Phi_F\\
&&=  (
  (
 (
F_{23}  (
 (\alpha \otimes \Delta )(F) ) 
  \Phi )  
  (
 (
 (\Delta \otimes \alpha)(F^{-1})  
 (F_{12}^{-1}  
 F_{12}) ) 
 (\Delta \otimes \alpha)(F) )
 ) \\
 &&
  (
(\Delta \otimes \alpha) \Delta(\alpha^3(x))  ) 
  (
(\Delta \otimes \alpha)(F^{-1})  
 F_{12}^{-1} )
 )
 \\
&&=  ((
  (
 (
F_{23}  
 (\alpha \otimes \Delta )(F) ) 
  \Phi) ( 
  (\Delta \otimes \alpha)(F^{-1})  
 F_{12}^{-1}) )
 (F_{12}  
 (\Delta \otimes \alpha)(F) ))\\ 
 && (
(\Delta \otimes \alpha) \Delta(\alpha^3(x))  )
  (
(\Delta \otimes \alpha)(F^{-1})  
 F_{12}^{-1} )
 )
 \\
&&=(
  (
 (
F_{23}  (
 (\alpha \otimes \Delta )(F) )
  \Phi) ( 
  (\Delta \otimes \alpha)(F^{-1})  
 F_{12}^{-1}) )\\ &&
  (
(F_{12}  
 (\Delta \otimes \alpha)(F) ) 
  (
(\Delta \otimes \alpha) \Delta(\alpha^2(x))
  (
(\Delta \otimes \alpha)(F^{-1})  
 F_{12}^{-1} )
 )
 )
 \\
&&=  (
 (
F_{23}  (
 (\alpha \otimes \Delta )(F)  
  \Phi )
 )
  (
 (\Delta \otimes \alpha)(F^{-1})  
 F_{12}^{-1}) )\\&&
  (
(F_{12}  
 (\Delta \otimes \alpha)(F) ) 
  (
 (
(\Delta \otimes \alpha) \Delta(\alpha(x))
(\Delta \otimes \alpha)(F^{-1} ))
 F_{12}^{-1} )
 )
 \\
&&=  (
F_{23}  (
 (\alpha \otimes \Delta )(F)  
  \Phi )
  (
 (\Delta \otimes \alpha)(F^{-1})  
 F_{12}^{-1}) )
 )\\&& 
  (
F_{12}  
  (
(\Delta \otimes \alpha)(F)  
  (
 (
(\Delta \otimes \alpha) \Delta(x))   
(\Delta \otimes \alpha)(F^{-1} ))
 F_{12}^{-1} )
 )
 ).
 \end{eqnarray*}
 }}
 Finally, since $\Delta \otimes \alpha$ is an algebra morphism and by identification we have
 {\small{
 \begin{eqnarray*}
 (\alpha^3 \otimes \Delta_F) (\Delta_F(x))  
 \Phi_F
&=& \Phi_F  
 (
 (
F_{12}  
 ((\Delta \otimes \alpha) \Delta_F(x)) )
 F_{12}^{-1} )
\\
&=& \Phi_F  
 ((\Delta_F \otimes \alpha^3) \Delta_F(x)).
 \end{eqnarray*}
 }}

However, the second axiom is easy to check.  The two last axioms involve only $F$ and $\Phi$ which satisfy $\alpha^{\otimes 2}(F)=F$ and  $\alpha^{\otimes 3}(\Phi)=\Phi$. The Hom-associativity may be handled as the associativity. It turns out that the proof in \cite[Page 374]{Kabook}  works in this context.

\end{proof}

\section{Twisted quantum double, an example} \label{ExampleDrinfeldTwist}
We consider an example of non-trivial (braided) quasi-bialgebra which relates to the quantum double of the algebra $\K G$, where $G$ is the finite cyclic group $\Z_3$.  Using a normalized $3$-cocycle $\omega$ of $G$, we construct a $9$-dimensional vector space $D^{\omega}G$ endowed with a structure of quasi-bialgebra.  We then find quasi-bialgebra morphisms of $D^{\omega}G$ which allows us to construct  HQ-bialgebras.  

Now we recall the construction of $D^{\omega}G$ for any group $G$ in general  (see \cite{Kabook}) and then specify it to the case of $\Z_3$ using the classification of $3$-cocycles established in \cite{AM}, see also \cite{Natale}.\\
Let ${\omega}$ be a normalized $3$-cocycle on a group $G$, i.e.  $\omega: G \times G \times G \rightarrow \K-\{0\}$ such that 
$$\omega(x,y,z) \omega(tx,y,z)^{-1} \omega(t,xy,z) \omega(t,x,yz)^{-1} \omega(t,x,y)=1, \;\; \forall t, x, y, z \in G.$$    
The normalized condition is, $\omega(x,y,z) =1$ whenever $x,y$ or $z$ is equal $1$. 

 Consider a finite dimensional vector space denoted $D^{\omega}G$ with basis $\{e_gx\}_{(g,x)\in G \times G}$ indexed by $G \times G$.  Define a product on  $D^{\omega}G$ by $(e_gx)(e_hy)= \delta_{g,xhx^{-1}} \; \theta(g,x,y)\;e_g(xy)$ where $\delta_{g,xhx^{-1}} $ is the kronecker-delta and $$\theta(g,x,y)= \omega(g,x,y) \omega(x,y,(xy)^{-1}gxy)  \omega(x,x^{-1}gx,y)^{-1}.$$  It is easily checked that $1=\sum_{g \in G} e_g1$.  Observe that if $\omega(x,y,z)=1, \forall x,y,z=1$, then $D^{\omega}G$ is isomorphic to the quantum double $D(\K G)$ \cite[ IX4-33]{Kabook}.  Observe that in general the map from $\K G$ to $D^{\omega}G$ given by $x \mapsto \sum_{g \in G} e_gx$ is not a morphism of algebra but  the map $x \mapsto \sum_{g \in G} e_g1$ is.\\
Define a comultiplication on $D^{\omega}G$ by $\Delta( e_gx)= \sum_{uv=g} \gamma(x,u,v) e_ux \otimes e_vx$, where $\gamma(x,u,v)=\omega(u,v,x)\omega(x,x^{-1}ux,x^{-1}vx) \omega (u,x, x^{-1}vx)^{-1}$ .  The Counit by $\epsilon
 (e_gx)=\delta_{g,1}$.  \\Set also $\Phi=\sum_{x,y,z \in G}\omega(x,y,z)^{-1}e_x \otimes e_y \otimes e_z$ and $R=\sum_{g \in G} e_g \otimes (\sum_he_h)g$, then we have the following Drinfeld's Theorem.

\begin{theorem}
The tuple $(D^{\omega}G, \mu, \Delta, \epsilon
, \Phi, R)$ is a quasi-triangular quasi-bialgebra.
\end{theorem}
One has also
 \begin{lemma}
   If $G$ is an abelian group then $$\gamma(a,b,c)=\theta(a,b,c)=\omega(b,c,a)\omega(a,b,c)\omega(b,a,c)^{-1}.$$
   \end{lemma}

In the following we aim to describe the multiplication, comultiplication and the twist of $D^{\omega}G$ for the group $G=\Z_3$. The $3$-cocycles for the group $\Z_3$  generated by $x$ are described in the following proposition.
   
   \begin{proposition}\cite{AM}
   Let $\Z_3$ be the multiplicative cyclic group generated by $x$, then every $3$-cocycle $\omega: \Z_3 \times \Z_3 \times \Z_3 \rightarrow \K$ has the form $\omega(1,u,v)=1=\omega(u,1,v)=\omega(u,v,1)$, $\omega(x,x,x)= p
$,  $\omega(x,x,x^2)= q
$, $\omega(x,x^2,x)=r^{-1} p
^{-1}$, $\omega(x,x^2,x^2)=r q
^{-1}$, $\omega(x^2,x,x)=r^{-1} p
  q
^{-1}$, $\omega(x^2,x, x^2)=r  p
$, $\omega(x^2,x^2,x)= q
 r^{-1} p
^{-1}$ and $\omega(x^2,x^2,x^2)=r p
^{-1}$, where $ p
$ and $ q
$ are non zero elements of $\K$ and $r$ is a cubic root of unity.
   \end{proposition}
     Straightforward computations give
 
\begin{lemma} We have
 \begin{eqnarray*}
  &&    \gamma(1,b,c)=\theta(1,b,c)=1, \; \gamma(b,1,c)=\theta(b,1,c)=1,\; \gamma(b,c,1)=\theta(b,c,1)=1 \\
  &&    \gamma(x,x,x)=\theta(x,x,x)= p
, \; \gamma(x,x,x^2)=\theta(x,x,x^2)=r^{-1} p
^{-1},\; 
  \\  &&  \gamma(x,x^2,x)=\theta(x,x^2,x)=r^{-1} p
^{-1} \;
     \gamma(x,x^2,x^2)=\theta(x,x^2,x^2)=r^{-1} p
^{-2},\\ &&    \gamma(x^2,x,x)=\theta(x^2,x,x)= p
^{2},\; \gamma(x^2,x,x^2)=\theta(x^2,x,x^2)=r p
, \\
&& \gamma(x^2,x^2,x)=\theta(x^2,x^2,x)=r  p
,\; \gamma(x^2,x,x^2)=\theta(x^2,x^2,x^2)=r p
^{-1}.
    \end{eqnarray*}
   \end{lemma}

The multiplication of $D\Z_3^\omega$ is given, with respect to the basis\\  $(1,e_1x,e_1x^2, e_x1,e_xx, e_xx^2, e_{x^2}1, e_{x^2}x, e_{x^2}x^2)$ where $1=e_11+e_1x+e_1x^2,$ by the following commutative non zero product
{\small{ 
 \begin{eqnarray*}
   &&   e_1x \cdot e_1x=e_1x^2, \; e_1x \cdot e_1x^2=e_11,   \; e_1x^2 \cdot e_1x^2=e_1x,   \; e_x1 \cdot e_x1=e_1x,\\ &&    e_x1 \cdot e_xx=e_xx,  \; e_x1 \cdot e_xx^2=e_xx^2, 
         \; e_xx \cdot e_xx=  p
 e_xx^2,     
        \;e_xx^2 \cdot e_xx^2= r^{-1} p
^{-2} e_x1,\\ &&  e_{x^2}1 \cdot e_{x^2}1= e_{x^2}1, 
         \;e_{x^2}1 \cdot e_{x^2}x= e_{x^2}x,  
            \;e_{x^2}1 \cdot e_{x^2}x^2= e_{x^2}x^2, \\ && 
        e_{x^2}x \cdot e_{x^2}x=  p
^2 e_{x^2}x^2, \;e_{x^2}x \cdot e_{x^2}x^2= r  p
 e_{x^2}1,
       \;e_{x^2}x^2 \cdot e_{x^2}x^2=r  p
^{-1}  e_{x^2}x. 
      \end{eqnarray*}
      }}
      The comultiplication is defined by  
\begin{eqnarray*}
\Delta(1)&=&1\otimes 1,\\
\Delta(e_1x)&=& e_1x \otimes e_1x +r^{-1} p
^{-1}(e_{x^2}x \otimes e_xx + e_xx\otimes e_{x^2}x),\\
\Delta(e_1x^2)&=& e_1x^2 \otimes e_1x^2 +r p
 (e_{x^2}x^2 \otimes e_xx^2 + e_xx^2 \otimes e_{x^2}x^2),\\
\Delta(e_x1)&=& e_11 \otimes e_x1 +e_{x}1 \otimes e_11 + e_{x^2}1 \otimes e_{x^2}1,\\
\Delta(e_xx)&=& e_1x \otimes e_xx +e_{x}x \otimes e_1x + r^{-1} p
^{-2}e_{x^2}x \otimes e_{x^2}x,\\
\Delta(e_xx^2)&=& e_1x^2 \otimes e_xx^2 +e_{x}x^2 \otimes e_1x^2 + r p
^{-1}e_{x^2}x^2 \otimes e_{x^2}x^2,\\
\Delta(e_{x^2}1)&=& e_11 \otimes e_{x^2}1 +e_{x^2}1 \otimes e_11 +e_x1 \otimes e_x1,\\
\Delta(e_{x^2}x)&=& e_1x \otimes e_{x^2}x +e_{x^2}x \otimes e_1x + p
 e_xx \otimes e_xx,\\
\Delta(e_{x^2}x^2)&=& e_1x^2 \otimes e_{x^2}x^2 +e_{x^2}x^2 \otimes e_1x^2 + p
^2 e_xx^2 \otimes e_xx^2.
\end{eqnarray*}
\noindent
The counit is given by $\epsilon
(e_11)=\;\epsilon
(e_1x)= \epsilon
(e_1x^2)=\epsilon
(1)=1$ and the non-specified values are zero. 
    
\begin{eqnarray*}
    \Phi&=& p
^{-1}e_x1 \otimes e_x1 \otimes e_x1 +  q
^{-1}e_x1 \otimes e_x1 \otimes e_{x^2}1 + r  p
 e_x1 \otimes e_{x^2}1 \otimes e_x1\\ &+& r^{-1}  q
 e_x1 \otimes e_{x^2}1 \otimes e_{x^2}1 + r  p
  q
^{-1}e_{x^2}1 \otimes e_x1 \otimes e_x1 +\\
   &+&  r^{-1}  p
^{-1}e_{x^2}1 \otimes e_{x^2}1 \otimes e_{x^2}1 + e_11 \otimes e_11 \otimes e_11 + \\ &+&\sum_{u,v \in \{x,x^2\}} ( e_11 \otimes e_u1 \otimes e_v1 + e_u1 \otimes e_11 \otimes e_v1 + e_u1 \otimes e_v1 \otimes e_11 ) + \\
&+&\sum_{u \in \{x,x^2\}} ( e_11 \otimes e_11 \otimes e_u1 + e_11 \otimes e_u1 \otimes e_u1 + e_u1 \otimes e_11 \otimes e_11 )    
\end{eqnarray*}
\noindent
Now we construct  morphisms $f:=(a_{i,j})_{1 \leq i,j \leq 9}$ of $D^{\omega}G$ satisfying the conditions $f(1)=1, \;f^{\otimes 3}(\Phi)=\Phi, f \circ \mu=\mu \circ f^{\otimes 2}, f^{\otimes 2}\Delta=\Delta \circ f$ and $\epsilon
 \circ f=\epsilon
$. 

 We assume in the sequel that the ground field $\K=\mathbb{C}$. Direct computations leads to the following solutions:\\
\noindent
Let $\xi=e^{i \pi/3}$ (cubic root of $-1$), then the following map $f: D^{\omega}\Z_3 \rightarrow  D^{\omega}\Z_3$ is quasi-bialgebra morphisms (over complex numbers).

\begin{example}\label{example1}
\begin{eqnarray*}
   f(1)=1, \;f(e_1x)=e_1x, \;f(e_1x^2)=e_1x^2, \; f(e_x1)=e_x1, \; 
   f(e_xx)=-\xi e_xx, \\  \; f(e_xx^2)=\xi^2 e_xx^2,  f(e_{x^2}1)=e_{x^2}1,  \; f(e_{x^2}x)= \xi^2e_{x^2}x,  \; f(e_{x^2}x^2)=-\xi e_{x^2}x^2.
  \end{eqnarray*}
  \end{example}  
  \begin{example} 
\begin{eqnarray*}
  g(1)=1, \;g(e_1x)=e_1x, \;g(e_1x^2)=e_1x^2, \; g(e_x1)=e_x1, \; g(e_xx)=\xi^2 e_xx,\\  \; g(e_xx^2)=- \xi e_xx^2, 
  g(e_{x^2}1)=e_{x^2}1,  \; g(e_{x^2}x)= - \xi e_{x^2}x,  \; f(e_{x^2}x^2)=\xi ^2i e_{x^2}x^2.
  \end{eqnarray*} 
  \end{example} 
According to twisting principle of quasi-bialgebra (Theorem \ref{Thm:TwistHQbialgra}), we construct using Example \ref{example1}, the following HQ-bialgebra    $D^{\omega}_{\xi}\Z_3$ defined on the vector space $D^{\omega}\Z_3$ with respect to the previous basis as 

The multiplication is given by
{\small{
\begin{eqnarray*}
&&  \mu_{\xi}(e_1x,e_1x)=e_1x^2, \; \mu_{\xi}(e_1x,e_1x^2)=1-e_x1-e_{x^2}1, \; \mu_{\xi}(e_1x^2,e_1x^2)=e_1x, \\ && \mu_{\xi}(e_x1,e_x1)=e_x1, \;
 \mu_{\xi}(e_x1,e_xx)=- \xi e_xx, \;  \mu_{\xi}(e_x1,e_xx^2)= \xi^2 e_xx, \\ && \;  \mu_{\xi}(e_xx,e_xx)= \xi^2  p
 e_xx^2, \; \mu_{\xi}(e_xx,e_xx^2)= r^{-1} p
 ^{-1} e_x1, \\ &&
 \mu_{\xi}(e_xx^2,e_xx^2)= - \xi r^{-1} p
^{-1} e_xx,  \;  \mu_{\xi}(e_{x^2}1,e_{x^2}1)=e_{x^2}1,  \;  \mu_{\xi}(e_{x^2}1,e_{x^2}x)=\xi ^2e_{x^2}x, \\ &&   \mu_{\xi}(e_{x^2}1,e_{x^2}x^2)=- \xi e_{x^2}x^2,\;  
   \mu_{\xi}(e_{x^2}x,e_{x^2}x)=- \xi  p
 ^2 e_{x^2}x^2, \\ && \;   \mu_{\xi}(e_{x^2}x,e_{x^2}x^2)=r  p
 e_{x^2}x^2,  \;   \mu_{\xi}(e_{x^2}x^2,e_{x^2}x^2)=\xi^2 r   p
 ^{-1} e_{x^2}x. 
    \end{eqnarray*}  
    }}
    \noindent
The comultiplication is defined as 
{\small{
\begin{eqnarray*}
&& \Delta(1)=1\otimes 1,\\
&& \Delta(e_1x)= e_1x \otimes e_1x +r^{-1} p
^{-1}(e_{x^2}x \otimes e_xx + e_xx\otimes e_{x^2}x),\\
&& \Delta(e_1x^2)= e_1x^2 \otimes e_1x^2 +r p
 (e_{x^2}x^2 \otimes e_xx^2 + e_xx^2 \otimes e_{x^2}x^2),\\
&& \Delta(e_x1)= 1 \otimes e_x1 +e_{x}1 \otimes 1 -2 e_x1 \otimes e_x1-e_{x^2}1 \otimes e_{x}1 -e_{x}1 \otimes e_{x^2}1 +e_{x^2}1 \otimes e_{x^2}1,\\
&& \Delta(e_xx)= -\xi (e_1x \otimes e_xx +e_{x}x \otimes e_1x + r^{-1} p
^{-2}e_{x^2}x \otimes e_{x^2}x),\\
&& \Delta(e_xx^2)= \xi ^2(e_1x^2 \otimes e_xx^2 +e_{x}x^2 \otimes e_1x^2 + r p
^{-1}e_{x^2}x^2 \otimes e_{x^2}x^2),\\
&& \Delta(e_{x^2}1)= 1 \otimes e_{x^2}1 +e_{x^2}1 \otimes 1 -e_x1 \otimes e_{x^2}1 -e_{x^2}1 \otimes e_x1 -2 e_{x^2}1 \otimes e_{x^2}1  +e_x1 \otimes e_x1,\\
&& \Delta(e_{x^2}x)= \xi ^2 (e_1x \otimes e_{x^2}x +e_{x^2}x \otimes e_1x + p
 e_xx \otimes e_xx),\\
&& \Delta(e_{x^2}x^2)= - \xi (e_1x^2 \otimes e_{x^2}x^2 +e_{x^2}x^2 \otimes e_1x^2 + p
^2 e_xx^2 \otimes e_xx^2).
\end{eqnarray*}
}}
\noindent
The counit is the same and is given by $\epsilon
(e_11)=\;\epsilon
(e_1x)= \epsilon
(e_1x^2)=\epsilon
(1)=1$ and the non-specified values are zero.   The map $\Phi$ is also the same.  
The twist map is  $f$ which is given in Example \ref{example1}.

Nevertheless, the new multiplication is no longer associative as can be seen from the following calculation: $(e_x1 \cdot  e_xx)\cdot e_xx= p
 e_xx^2$ and $e_x1 \cdot (e_xx \cdot e_xx)=- \xi  p
 e_xx^2$.
    

\bibliographystyle{amsplain}
\providecommand{\bysame}{\leavevmode\hbox to3em{\hrulefill}\thinspace}
\providecommand{\MR}{\relax\ifhmode\unskip\space\fi MR }
\providecommand{\MRhref}[2]{%
  \href{http://www.ams.org/mathscinet-getitem?mr=#1}{#2}
}
\providecommand{\href}[2]{#2}

\end{document}